\documentclass[11pt]{amsart}

\usepackage[T1]{fontenc}
\usepackage[utf8]{inputenc}
\usepackage{textcomp}
\usepackage{lmodern}
\usepackage{microtype}

\usepackage{amssymb}
\usepackage{amsthm}
\usepackage{amscd}
\usepackage{xcolor}

\usepackage{hyperref}

\usepackage{mathscinet}

\newtheorem{theorem}{Theorem}[section]

\newtheorem{lemma}[theorem]{Lemma}
\newtheorem{prop}[theorem]{Proposition}
\newtheorem{cor}[theorem]{Corollary}

\theoremstyle{definition}
\newtheorem{definition}[theorem]{Definition}

\allowdisplaybreaks[2]

 \newcommand{\RCA}{\ensuremath{\mathsf{RCA_0}}}
 \newcommand{\ATR}{\ensuremath{\mathsf{ATR_0}}}
\DeclareMathOperator{\HYPP}{\mathit{HYP}}
\newcommand{\HYP}{\ensuremath{\HYPP}}

\newcommand{\CDPB}{\mathsf{CD}\text{-}\mathsf{PB}}
\newcommand{\concat}{^\frown}
\newcommand{\dom}{\operatorname{dom}}

\begin{document}

\title{Borel Combinatorics Fail in \HYP}
\author{Henry Towsner}
\address {Department of Mathematics, University of Pennsylvania, 209 South 33rd Street, Philadelphia, PA 19104-6395, USA}
\email{htowsner@math.upenn.edu}
\urladdr{\url{http://www.math.upenn.edu/~htowsner}}

\author{Rose Weisshaar}
\address {Department of Mathematics, University of Pennsylvania, 209 South 33rd Street, Philadelphia, PA 19104-6395, USA}
\email{roseweis@math.upenn.edu}
\urladdr{\url{http://www.math.upenn.edu/~roseweis}}

\author{Linda Westrick}
\address {Department of Mathematics, Penn State University, University Park, PA 16802, USA}
\email{westrick@psu.edu}
\urladdr{\url{http://www.personal.psu.edu/lzw299/}}

\thanks{Towsner was partially supported by NSF grant DMS-1600263.  
Westrick was partially supported by NSF grant DMS-1854107 and by the Cada R. and Susan Wynn Grove Early Career Professorship in Mathematics.}

\date{\today}

\begin{abstract}  We characterize the completely determined Borel subsets of \HYP{} as 
exactly the $\Delta_1(L_{\omega_1^{ck}})$ subsets of \HYP{}.  As a result, \HYP{}
believes there is a Borel well-ordering of the reals, that
the Borel Dual Ramsey Theorem fails, and that
every Borel $d$-regular bipartite graph has a Borel perfect matching, among other 
examples.  Therefore,
the Borel Dual Ramsey Theorem and several theorems of descriptive 
combinatorics are not theories of hyperarithmetic analysis.  In the case of the 
Borel Dual Ramsey Theorem, this
answers a question of Astor, Dzhafarov, Montalb\'an, Solomon \& the third author.
\end{abstract}

\maketitle

\section{Introduction}
Theorems about Borel sets are often proved using arguments which appeal to some property of Borel sets, rather than proceeding by transfinite recursion on the 
structure of the set directly.  Examples include category arguments, measure arguments, and Borel determinacy arguments.  When a theorem has been proved using one of these methods, it is natural to wonder if there are essentially different proofs.  Reverse Mathematics provides a framework for answering this kind of curiosity.  In this paper we consider the Reverse Mathematics strength of 
several such theorems, one from Ramsey theory and the rest from descriptive combinatorics. 

The Reverse Math strength of the Dual Ramsey Theorem \cite{CarlsonSimpson1984}
 has been the topic of several 
papers \cite{Simpson1986,MillerSolomon2004,DFSW,ADMSW}.  
In this theorem, one starts with a ``nice'' 
coloring of the space of partitions of $\omega$ into $k$
pieces, and the theorem guarantees a partition of $\omega$ into infinitely many pieces,
all of whose $k$-piece coarsenings have the same color.  
When ``nice'' means Borel, in \cite{ADMSW} it 
was shown that the Borel Dual Ramsey Theorem for 3-partitions 
follows from $\CDPB+\mathsf{ACA}_0^+$, where $\CDPB$ is the statement
``every completely determined Borel set has the property of Baire''.\footnote{In fact,
since $\CDPB$ implies $\mathsf L_{\omega_1,\omega}\text - \mathsf{CA}_0$,
 the Borel Dual Ramsey Theorem for 3-partitions follows from $\CDPB + \mathsf{\Sigma^1_1\text-IND}$.}  ``Completely determined'' refers to a restricted 
 way in which Borel sets can be encoded; see Section \ref{sec:preliminaries} for details.
 This reflects the fact that the proof of the Borel Dual Ramsey Theorem 
uses a category argument.  In fact, the theorem is also true for 
colorings which have the property of Baire \cite{ProemelVoigt1985}.

In \cite{ADMSW}, 
it was left as an open question whether the Borel Dual Ramsey Theorem 
is a statement of hyperarithmetic analysis.  If it were, 
it would imply that the category argument in the usual proof is 
not essential, because $\CDPB$ fails in \HYP{} \cite{ADMSW}, while
by definition every statement of hyperarithmetic analysis holds in \HYP.

\begin{theorem} For any finite $k,\ell \geq 2$,
The Borel Dual Ramsey Theorem for $k$-partitions and $\ell$ colors fails in \HYP.
Therefore, the Borel Dual Ramsey Theorem is not a statement of
hyperarithmetic analysis.
\end{theorem}

It remains open whether the Borel Dual Ramsey Theorem implies $\CDPB$.

Our second motivation comes from the area of descriptive combinatorics. 
Using the axiom of choice, any $d$-regular bipartite graph has a perfect matching,
and any acyclic graph has a 2-coloring. 
However, 
if we restrict attention to Borel perfect matchings and Borel colorings, 
the matching may no longer exist or the needed number of colors may increase.
This area is surveyed in \cite{KechrisMarks}.

Marks has shown that for all $d\geq 2$, there is a $d$-regular acyclic Borel graph with no $d$-coloring, and a $d$-regular acyclic Borel bipartite graph with no Borel perfect matching \cite{Marks16}.  
The proofs use a Borel determinacy argument, in contrast to the more typical use 
of measure and category arguments to prove theorems in this area.
In a talk given at the ASL Annual 
Meeting in Macomb in 2018, Marks wondered whether such a big 
hammer was really needed, and asked for the Reverse Mathematics strength
of the perfect matching theorem.  Kun recently gave a partial answer by 
providing a measure-theoretic proof of the perfect matching theorem \cite{Kun2021}.  We show that no statement of hyperarithmetic 
analysis is strong enough for either theorem.
\begin{theorem}
In \HYP, every completely determined Borel $d$-regular graph with no odd cycles 
has a completely 
determined Borel perfect matching and a completely determined Borel 2-coloring.
\end{theorem}
Statements of hyperarithmetic analysis are among the weakest axioms 
strong enough to make sense of Borel sets.  It would be interesting to know 
whether Marks' $d$-coloring theorem can be proved via a measure or category 
argument, two methods which suffice for many theorems of descriptive 
combinatorics.  We do not take on that question here, but 
for a brief discussion of how it can be formalized,
 see the end of Section \ref{subsection:graphs}.

Both results above are consequences of the main theorem of this paper,
characterizing those 
subsets of \HYP{} which \HYP{} believes are completely determined Borel.
Recall that $L_{\omega_1^{ck}} \cap 2^\omega =\HYP$.

\begin{theorem}\label{thm:main}
  For any $A \subseteq \HYP$, the following are equivalent.
  \begin{enumerate}
  \item  There is a completely determined Borel
code for $A$ in \HYP.
\item There is a determined Borel code for $A$ in \HYP.
\item  $A$ is $\Delta_1(L_{\omega_1^{ck}})$. 
\end{enumerate}
\end{theorem}

Definitions of completely determined and determined Borel codes are given 
in Section \ref{sec:preliminaries}.
The proof makes essential use of non-standard Borel codes and 
the method of decorating trees which was introduced in \cite{ADMSW}.

In both the Borel Dual Ramsey Theorem and Marks' theorems, some 
restriction on the coloring and/or perfect matching is known to be necessary; the failure of these theorems without the Borel condition is witnessed by straightforward choice arguments.  Strangely, the failure of these theorems in \HYP{} is witnessed by essentially the same choice arguments, albeit in a more technical form.  This is possible due to the following pathology of Borel sets in \HYP.
\begin{theorem}
In \HYP, there is a completely determined Borel well-ordering of the reals.
\end{theorem}
We use similar methods to construct choice-flavored counterexamples in \HYP{} to some other theorems of descriptive combinatorics, such as those concerning the prisoner hat problem
and various vertex and edge coloring theorems for $d$-regular graphs.

Having recreated some choice-flavored constructions, we asked how 
reliably Borel constructions in \HYP{} mimic choice constructions in the real world. We find that the analogy is not perfect, as the following result shows.
\begin{theorem}
In \HYP, there is a completely determined 
Borel acyclic graph where each vertex has degree at most 2, but which has no 
completely determined Borel 2-coloring.
\end{theorem}

We give the preliminaries in Sections \ref{sec:preliminaries} and \ref{sec:decorating},
the latter of which is devoted entirely to the method of decorating trees, making
this paper self-contained for readers already familiar with Reverse Mathematics 
and hyperarithmetic theory.  The main result characterizing the completely determined
Borel sets in \HYP~ is given in Section \ref{sec:characterization}. 
Section \ref{sec:applications} contains all of the applications.

We thank Andrew Marks for alerting us to the recent developments and status of open questions in this area, 
and the anonymous referee for a careful reading and many small improvements.  Of course,
any mistake that remains is due to the authors.

\section{Preliminaries}\label{sec:preliminaries}

We denote elements of $\omega^{<\omega}$ by $\sigma,\tau,\eta,\nu$.  We write $\sigma\preceq \tau$ to indicate that $\sigma$ is an initial segment of $\tau$, and write $\sigma\prec \tau$ if $\sigma$ is a proper initial segment of $\tau$.  We write $\sigma^\frown\tau$ for the concatenation of $\sigma$ and $\tau$.  We write $\sigma^\frown n$ as an abbreviation for $\sigma^\frown\langle n\rangle$.

Throughout, we assume familiarity with hyperarithmetic theory and reverse mathematics.
A standard reference for the former is \cite{sacks} and for the latter, \cite{sosa}.
We are primarily interested in considering notions within the second order model \HYP; this is the model of second-order arithmetic in which the natural numbers are interpreted by the usual natural numbers but the only sets present are the hyperarithmetic sets.

We write $\mathcal{O}^\ast$ for the set of ordinal notations in \HYP, and $<_\ast$ for the 
computable partial order comparing those notations.  
We will use $\alpha,\beta,\gamma,\delta$ for elements of $\mathcal{O}^\ast$ 
These notations represent the 
ordinals of \HYP{} because $\alpha\in\mathcal{O}^\ast$ if and only if there is no hyperarithmetic $<_\ast$-descending sequence below $\alpha$.  
It is well-known that there are elements $\alpha$ in $\mathcal{O}^\ast$ such that $<_\ast$ is, in fact, ill-founded below $\alpha$, but no descending sequence is hyperarithmetic.  As usual, we write $\mathcal{O}$ for the subset of $\mathcal{O}^\ast$ consisting of actual ordinals---that is, $\alpha\in\mathcal{O}$ if and only if there is no $<_\ast$-descending sequence below $\alpha$.

When care is needed in the use of notations, we use the standard notation $H_\alpha$ 
to refer to the set obtained by taking jumps along the notation $\alpha$.  When 
we only need to refer to a set in the same $\leq_T$-degree as $H_\alpha$,
we use the notation $\emptyset^\alpha$.   We often abuse notation by identifying ordinal 
notations with the ordinals they represent, writing for example
$\alpha+k$, or $\alpha +O(1)$ to refer to an ordinal which is a finite successor of $\alpha$.

\begin{definition}
  A \emph{tree} is a subset of $\omega^{<\omega}$ closed under initial segments.  When $T$ is a tree, we write $T_n=\{\sigma\mid \langle n\rangle^\frown\sigma\in T\}$.

  A \emph{labeled Borel code} is a well-founded tree $T\subseteq\omega^{<\omega}$ together with a function $\ell$ whose domain is $T$ and such that:
  \begin{itemize}
  \item for each interior node $\sigma$ of $T$, $\ell(\sigma)$ is either $\bigcup$ or $\bigcap$,
  \item for each leaf $\eta$ of $T$, $\ell(\eta)$ is a standard code for a clopen subset of $2^\omega$.
  \end{itemize}

  When $\ell(\sigma)=\bigcup$, we call $\sigma$ a \emph{union node}, and when $\ell(\sigma)=\bigcap$, we call $\sigma$ an \emph{intersection node}.
\end{definition}


We will be considering Borel codes in \HYP---that is, $T$ and $\ell$ are themselves hyperarithmetic, and there is no hyperarithmetic descending sequence in $T$.  Equivalently, $T$ has a height in $\mathcal{O}^\ast$.

We can ask for codes which make this ordinal height explicit.
\begin{definition}
  Let $\alpha\in\mathcal{O}^\ast$.  If $T\subseteq\omega^{<\omega}$ and $\rho:T\rightarrow \{\beta  \in \mathcal O^\ast : \beta \leq_\ast \alpha\}$, we say that $\rho$ \emph{ranks} $T$ if for all $\sigma$ and $n$ such that $\sigma^\frown\langle n\rangle\in T$, we have $\rho(\sigma^\frown n)<_\ast\rho(\sigma)$.  We say $T$ is \emph{$\alpha$-ranked} by $\rho$.  We call $\rho(\langle\rangle)$ the \emph{rank} of $T$.
\end{definition}

When $T,\ell$ is a true Borel code, it encodes a subset $|T|$ of $2^\omega$.  Namely:
\begin{itemize}
\item if $\langle\rangle$ is a leaf, $|T_{\langle\rangle}|$ is the clopen set coded by $\ell(\langle\rangle)$,
\item if $\ell(\langle\rangle)=\bigcup$, $T$ codes $\bigcup_n |T_n|$,
\item if $\ell(\langle\rangle)=\bigcap$, $T$ codes $\bigcap_n |T_n|$.
\end{itemize}

To make this precise in a model of second order arithmetic, we need the notion of an evaluation map.
\begin{definition}
  When $T$ is a labeled Borel code and $X\in 2^\omega$, an \emph{evaluation map} for $X\in T$ is a function $f:T\rightarrow\{0,1\}$ such that:
  \begin{itemize}
  \item if $\eta$ is a leaf, $f(\eta)=1$ if and only if $X$ is in the clopen set coded by $\ell(\eta)$,
  \item if $\sigma$ is a union node, $f(\sigma)=1$ if and only if $f(\sigma^\frown n)=1$ for some $n\in\omega$,
  \item if $\sigma$ is an intersection node, $f(\sigma)=1$ if and only if $f(\sigma^\frown n)=1$ for all $n\in\omega$.
  \end{itemize}
  We say $X$ is in the set coded by $T$, denoted $X\in|T|$, if there is an evaluation map $f$ for $X$ in $T$ such that $f(\langle\rangle)=1$,  We write $X\not\in|T|$ if there is an evaluation map $f$ for $X$ in $T$ such that $f(\langle\rangle)=0$.
\end{definition}
The statement ``for every labeled Borel code $T$ there is an $X$ which has an evaluation map in $T$'' is equivalent to \ATR{} \cite[Theorem 6.9]{DFSW}.  In particular, in \HYP{} there are labeled Borel codes 
for which no evaluation maps exist for any $X$.
In \cite{ADMSW} this is addressed by introducing the notion of a completely determined Borel code.
\begin{definition}
  A labeled Borel code $T$ is \emph{completely determined} if every $X\in 2^\omega$ has an evaluation map in $T$.
\end{definition}
Note that $\RCA$ suffices to prove that any two evaluation maps must agree. 
For if two evaluation maps disagree at some node $\sigma \in T$, then 
they must also disagree at some longer node $\sigma\concat n \in T$.  Therefore, 
from two disagreeing evaluation maps, we may recursively construct a
path through $T$, violating that $T$ is well-founded.  Formally, this argument uses 
\cite[Theorems II.3.4, II.3.5]{sosa}.

A related notion, named but not studied in \cite{ADMSW}, is a determined Borel code.
Considering a Borel code as a game played by a $\bigvee$ player against a
$\bigwedge$ player in the sense of \cite{Blackwell}, 
the code is called determined if for every $X$, one of the players has a 
winning strategy in the game.
\begin{definition}
A labeled Borel code $T$ is \emph{determined} if for every $X \in 2^\omega$, there is 
a function $f:\subseteq T \rightarrow \{0,1\}$, called a \emph{winning strategy for $X$ in $T$}, 
such that 
\begin{itemize}
\item If $\sigma$ is a leaf and $f(\sigma)$ is defined, then $f(\sigma)=1$ if and only if $X$ 
is in the clopen set coded by $\ell(\sigma)$.
\item If $\sigma$ is a union node, $f(\sigma) = 1$ implies there is some $n\in \omega$ such that $f(\sigma\concat n) = 1$, and $f(\sigma) = 0$ implies for all $n \in \omega$, if $\sigma\concat n \in T$ then $f(\sigma\concat n) = 0$.
\item If $\sigma$ is an intersection node, $f(\sigma) = 0$ implies there is some $n \in \omega$ 
such that $f(\sigma\concat n) = 0$, and $f(\sigma)=1$ implies that for all $n \in \omega$,
if $\sigma\concat n \in T$ then $f(\sigma \concat n) = 1$.
\item $f(\langle \rangle)$ is defined.
\end{itemize}
\end{definition}

It can happen that a Borel code is determined without being completely determined. 
For example, in \HYP, let $T$ be a Borel code which is not completely determined. 
Then the set $\emptyset \cap |T|$, written as a Borel code with $\bigcap$ at the root,
is determined but not completely determined in \HYP.

Given a Borel code $T$, we define a code for its complement as follows.
\begin{definition}
If $T$ is a Borel code, let $\neg T$ denote the Borel code which uses
the same tree, but modifies the labeling function as follows.  Change $\bigcap$ to 
$\bigcup$ and vice versa at all interior nodes, and at each leaf replace the 
coded clopen set with its clopen complement.
\end{definition}

It is clear that if $f$ is an evaluation map for $X$ in $T$, then $1-f$ is an 
evaluation map for $X$ in $\neg T$, and thus regardless of the model, 
$X \in |T|$ if and only if
$X\not\in |\neg T|$.

\section{Decorating Trees}\label{sec:decorating}

The main method we use is a construction from \cite{ADMSW} which takes a tree $T$ and ``decorates it'' with additional nodes to create a new Borel code.  When we perform this decoration properly, the resulting Borel code will be completely determined in \HYP.
The results of this section were essentially proved in 
\cite{ADMSW}, but to keep this paper 
self-contained, we present them here with more streamlined notation and proofs.

\begin{definition}
  Let $\alpha\in\mathcal{O}^\ast$ and let $T$ be a labeled Borel code $\alpha$-ranked by $\rho$.  Suppose $\mathcal P$ and $\mathcal N$ are two countable sets of $\alpha$-ranked labeled Borel codes.  We define the \emph{decoration} of $T$ by $\{\mathcal P, \mathcal N\}$, denoted $\mathrm{Decorate}(T,\mathcal P, \mathcal N)$, recursively by:
  \begin{itemize}
  \item if $T$ is a leaf, $T$ is unchanged,
  \item otherwise, the children of $\langle\rangle$ in $\mathrm{Decorate}(T,\mathcal P, \mathcal N)$ are given by:
    \begin{itemize}
    \item for each child $T_n$ of $T$, the tree $\mathrm{Decorate}(T_n,\mathcal P, \mathcal N)$ is a child,
    \item if $\langle\rangle$ is a union node, for each $P \in \mathcal P$ where $P$ has rank $<_\ast\rho(\langle\rangle)$, the node $\mathrm{Decorate}(P, \mathcal P, \mathcal N)$ is a child, and
    \item if $\langle\rangle$ is an intersection node, for each $N \in \mathcal N$ where $N$ has rank $<_\ast\rho(\langle\rangle)$, the node $\mathrm{Decorate}(\neg N, \mathcal P, \mathcal N)$ is a child.
    \end{itemize}
  \end{itemize}
\end{definition}

Since $T$ and all elements of $\mathcal P \cup \mathcal N$ are $\alpha$-ranked,
the restriction on the ranks of $P$ and $N$ ensures that 
$\mathrm{Decorate}(T,\mathcal P, \mathcal N)$ is also
$\alpha$-ranked.

\begin{lemma}\label{lem:wf}
  If $\alpha\in\mathcal{O}$, $X\not\in |Q|$ for every $Q \in \mathcal P \cup \mathcal N$ of rank less than $\alpha$, and $T$ is ranked in $\alpha$ then $X\in|\mathrm{Decorate}(T,\mathcal P, \mathcal N)|$ if and only if $X\in|T|$.
\end{lemma}
\begin{proof}
  By induction on $\alpha$.  Let $g$ be the evaluation map for $X$ in $T$ and $h$ the evaluation map for $X$ in $\mathrm{Decorate}(T,\mathcal P, \mathcal N)$---since $\alpha$ is an actual ordinal, both exist and are unique.

  If $T$ is a leaf, this is immediate.  Otherwise, consider the children of the root in $\mathrm{Decorate}(T,\mathcal P, \mathcal N)$.  Say $\langle\rangle$ is a union node.  If there is some child $T_n$ in $T$ which $g$ assigns to $1$, then by the inductive hypothesis, $h$ must assign $1$ to the corresponding child node $\mathrm{Decorate}(T_n,\mathcal P, \mathcal N)$ in $\mathrm{Decorate}(T,\mathcal P, \mathcal N)$, so $h(\langle\rangle)=1$.  Otherwise, $g$ assigns $0$ to every child of $\langle\rangle$ in $T$.  Every child of $\langle\rangle$ in $\mathrm{Decorate}(T,\mathcal P, \mathcal N)$ is either of the form $\mathrm{Decorate}(T_n,\mathcal P, \mathcal N)$ or $\mathrm{Decorate}(P,\mathcal P, \mathcal N)$; by the inductive hypothesis and the assumption that $X\not\in |P|$, $h$ assigns $0$ to both kinds of children, so $h(\langle\rangle)=0$.

  The intersection case is symmetric: if $g$ assigns $0$ to any child $T_n$ of $\langle\rangle$ then, by the inductive hypothesis, $h$ must assign $0$ to the corresponding child node $\mathrm{Decorate}(T_n,\mathcal P, \mathcal N)$ in $\mathrm{Decorate}(T,\mathcal P, \mathcal N)$, so $h(\langle\rangle)=0$.  If $g$ assigns $1$ to every child of $\langle\rangle$ in $T$ then, since the children of $\langle\rangle$ in $\mathrm{Decorate}(T,\mathcal P, \mathcal N)$ are either of the form $\mathrm{Decorate}(T_n,\mathcal P, \mathcal N)$ or $\mathrm{Decorate}(\neg N,\mathcal P, \mathcal N)$; by the inductive hypothesis and the assumption that $X\in |\neg N|$, $h$ assigns $1$ to both kinds of children, so $h(\langle\rangle)=1$.
\end{proof}

We will be interested in the situation where we carry this operation out in \HYP.  Note that when $\alpha\in\mathcal{O}^\ast$, $T$ is in \HYP, and the collections 
$\mathcal P$ and $\mathcal N$ are enumerable in \HYP{} (that is, \HYP{}  
contains sequences $\langle P_n\rangle_{n\in\omega}$ and $\langle N_n\rangle_{n\in\omega}$ such that $\mathcal P = \{P_n: n \in \omega\}$ and $\mathcal N=\{N_n:n \in \omega\}$), 
then the labeled Borel code $\mathrm{Decorate}(T,\mathcal P, \mathcal N)$ is in \HYP{} as well.

Let $\mathcal P_\mathcal O$ denote the subset of $\mathcal P$ consisting of codes 
whose rank is well-founded, and similarly define $\mathcal N_\mathcal O$.  
The key result is the following:
\begin{theorem}\label{thm:main_decoration}
Let $\alpha\in\mathcal{O}^\ast\setminus\mathcal{O}$.
Suppose that $\mathcal P$ and $\mathcal N$ are countable collections of $\alpha$-ranked
decorations, enumerable in \HYP{}, such that
for each $X\in\HYP$, there is a unique $Q \in \mathcal P_\mathcal O \cup \mathcal N_\mathcal O$ with $X \in |Q|$.
Then there is a computable tree $T$ such that in \HYP{},
$\mathrm{Decorate}(T, \mathcal P, \mathcal N)$ is completely
determined and
$|\mathrm{Decorate}(T, \mathcal P, \mathcal N)|=\bigcup_{P \in \mathcal P_\mathcal O}|P|$.
\end{theorem}
\begin{proof}
Let $T$ be the tree $\{\langle\rangle, \langle 1 \rangle\}$ where $\langle\rangle$ is a union node and $\rho(\langle\rangle)=\alpha$, while $\langle 1 \rangle$ is a 
leaf coding $\emptyset$ which has rank 0.
 
 For technical reasons, it will be convenient to assume that each element of 
 $\mathcal P$ has an intersection at its root.  This is a harmless assumption - 
 given any enumeration of $\mathcal P$, we may simply modify each code $P$
 in it, increasing its rank by one in order to add a new root which expresses a 
 trivial intersection whose only argument is $P$.  Increasing $\alpha$ by 1 as 
 well, this addition does not 
 endanger any of the hypotheses of the theorem.
  
  The key idea is this: given a hyperarithmetic set $X$, and the unique 
  $Q \in \mathcal P_\mathcal O \cup \mathcal N_\mathcal O$ 
  such that $X\in |Q|$,
  we can find a hyperarithmetic evaluation map for $X$
  in $\mathrm{Decorate}(T,\mathcal P, \mathcal N)$.  We can always find hyperarithmetic evaluation maps for the low-ranked parts of $\mathrm{Decorate}(T, \mathcal P, \mathcal N)$.  Since many high ranked nodes will have a decorated version of $Q$
  as a subtree, we can then systematically assign values of the 
  evaluation map to these nodes.

  So let $X$ be given and let $\gamma$ be the rank of $Q$.  Since $\mathrm{Decorate}(T, \mathcal P, \mathcal N)$ is hyperarithmetic and $\gamma\in\mathcal{O}$, there is a partially defined evaluation map $g_0$ defined on all nodes of $\mathrm{Decorate}(T, \mathcal P, \mathcal N)$ with rank $\leq\gamma$.  (Such a $g_0$ can be computed in slightly more than $\gamma$ jumps from $\mathrm{Decorate}(T, \mathcal P, \mathcal N)$.)


  Suppose $Q \in \mathcal P$.  We extend $g_0$ to an evaluation map $g$ on all of $\mathrm{Decorate}(T, \mathcal P, \mathcal N)$ as follows:
  \begin{itemize}
  \item If $\sigma$ is a union node with rank $>_\ast\gamma$, $g(\sigma)=1$.  Since one of the children of $\sigma$ is a copy of $\mathrm{Decorate}(Q, \mathcal P, \mathcal N)$, which, by Lemma \ref{lem:wf}, $g_0$ must assign $1$ to, this is a correct evaluation map.
  \item If $\sigma$ is an intersection node then consider the following set of
  descendants of $\sigma$:
  \begin{multline*}\qquad D_\sigma = \{ \tau\in \mathrm{Decorate}(T,\mathcal P,\mathcal N) : \tau \succ \sigma, \tau \text{ is a union or leaf},\\ \text{and for each $\nu$ with } 
  \tau \succ \nu \succ \sigma, \text{$\nu$ is an intersection}\}.\end{multline*}
  For each $\tau \in D_\sigma$, if $\rho(\tau)\leq_\ast \gamma$, then $\tau$ 
  is in the domain of $g_0$, so we know the 
  correct value for $\sigma$ based on $g_0$.  
  If $\rho(\tau) >_\ast \gamma$, then we shall assign
  $g(\tau) = 1$, so these nodes can be safely ignored, as they can 
  only help $X$ get into the intersection at $\sigma$.
  We assign $1$ to $\sigma$ if and only if 
  every $\tau \in D_\sigma$ of rank $\leq_\ast\gamma$ has been assigned 1 by 
  $g_0$ (as defined in the previous step).  This can be 
  done uniformly in one jump of $g_0$.
  \end{itemize}
  Therefore $g$ can be computed from $g_0$ in one more jump.  It is clear that 
  $g$ satisfies the definition of an evaluation map.
  Finally, $g$ assigns $1$ to $\langle\rangle$ because this is a union node of 
  rank $\alpha >_\ast \gamma$.

  The case where $Q \in \mathcal N$ is dual, with one small addition to the argument needed to verify the value of $g(\langle \rangle)$.  We extend $g_0$ to an evaluation map $g$ by:
  \begin{itemize}
  \item If $\sigma$ is an intersection node with rank $>_\ast\gamma$ then $g(\sigma)=0$.  Since $X \not\in |\neg Q|$ and one of the children is a copy of $\mathrm{Decorate}(\neg Q,\mathcal P, \mathcal N)$, this is a correct evaluation map by Lemma \ref{lem:wf}.
  \item If $\sigma$ is a union node with rank $>_\ast\gamma$, define $D_\sigma$ in a dual way to what was done above, swapping intersections and unions:
  \begin{multline*}\qquad D_\sigma = \{ \tau\in T : \tau \succ \sigma, \tau \text{ is an intersection or leaf},\\ \text{and for each $\nu$ with } 
  \tau \succ \nu \succ \sigma, \text{$\nu$ is a union}\}.\end{multline*}  
  Each $\tau \in D_\sigma$ of rank $\leq_\ast\gamma$  is in the domain of $g_0$.  If 
  any $\tau \in D_\sigma$ has rank $>_\ast\gamma$ then we shall have $g(\tau)=0$,
  so these nodes can be safely ignored, as they cannot help $X$ get 
  into the union at $\sigma$.  
  We assign $1$ to $\sigma$ if and only if some $\tau \in D_\sigma$ of rank $\leq_\ast\gamma$ has been assigned 1 by $g_0$.
  \end{itemize}
  Again, $g$ is an evaluation map which can be computed from $g_0$ in one more jump.   
  Now we wish to show that $g(\langle \rangle)=0$.  Consider the set $D_{\langle \rangle}$.  Because every element of $\mathcal P$ has an intersection at its root, and
  $\langle \rangle$ has only a single leaf child in $T$, every child of $\langle \rangle$ in 
  $\mathrm{Decorate}(T,\mathcal P, \mathcal N)$ is an intersection or 
  leaf node.  Therefore,
  $D_{\langle \rangle}$ is exactly the set of children of $\langle \rangle$, and these 
 all take the form $\mathrm{Decorate}(P,\mathcal P, \mathcal N)$ for some 
 $P \in \mathcal P$, plus the single leaf, which has been unchanged 
 by decoration.  For each non-leaf 
 child $\tau$ with rank $\leq_\ast \gamma$,  $X \not\in |P|$, and thus by Lemma 
 \ref{lem:wf}, $X \not \in |\mathrm{Decorate}(P,\mathcal P, \mathcal N)|$ and 
 $g_0(\tau) = 0$.  Therefore, $g(\langle \rangle) = 0$, as needed.
\end{proof}

\section{Characterization of Borel sets in \HYP}\label{sec:characterization}

Our main theorem is the following.  Considering G\"odel's constructible universe $L=\bigcup_{\mu \in \mathrm{Ord}} L_\mu$, recall that
$L_{\omega_1^{ck}} \cap 2^\omega =\HYP$.

\begin{theorem}\label{thm:main2}
  For any $A \subseteq\HYP$, the following are equivalent.
  \begin{enumerate}
  \item  There is a completely determined Borel
code for $A$ in \HYP.
\item There is a determined Borel code for $A$ in \HYP.
\item  $A$ is $\Delta_1(L_{\omega_1^{ck}})$. 
\end{enumerate}
\end{theorem}

Before proving this, recall that
for any $\Sigma_1$ formula $\theta(x)$ in the language of set theory, 
we have that $L_{\omega_1^{ck}}\models \theta(x)$ 
if and only if there is some $\alpha < \omega_1^{ck}$ such that $L_\alpha \models \theta(x)$.
Therefore, it will be useful to bound the complexity of deciding facts
about $L_\alpha$.  In short, it is well-known that 
$\emptyset^{\omega\cdot\alpha}$ can compute 
a presentation of $L_\alpha$, but we give a (rather standard) proof here,
 because we also need to take a little care with 
the ordinal notations when using this claim.
Specifically, we give an algorithm which computes a presentation 
of $L_\alpha$ given $H_{\omega\cdot \alpha}$, where
$\omega\cdot \alpha$ is the notation defined as follows.  Let 
$\omega\cdot \alpha = 3\cdot 5^{e(\alpha)}$, where $e$ is defined
recursively by
$$\phi_{e(\alpha)}(n) = \begin{cases} \omega\cdot \alpha_n & \text{ if } \alpha = \lim_n \alpha_n\\
\omega\cdot(\alpha-1) + n & \text{ if } \alpha \text{ is a successor.}\end{cases}$$
Here the ``$+n$'' in the second line is shorthand for a height $n$ tower of 2's.
Representing the notations for $\omega \cdot \alpha$ in this way gives us
a uniform procedure which finds, for each $\beta <_\ast \alpha$,
compatible notations $\omega \cdot \beta <_\ast \omega\cdot \alpha$.

\begin{prop}\label{prop:Lsays}
There is a computable procedure which, 
given $\alpha \in \mathcal O$ and $H_{\omega\cdot \alpha}$, 
returns a presentation $\Theta_\alpha$ of $L_\alpha$
(in the language of set theory, $\{\epsilon\}$).  
Furthermore, the procedure can be chosen so that the 
presentations have two nice properties:  
\begin{enumerate}
\item Whenever $\beta <_\ast \alpha$,
the restriction of $\Theta_\alpha$ to the domain of $\Theta_\beta$
is equal to $\Theta_\beta$ and is an $\epsilon$-initial segment 
of $\Theta_\alpha$.
\item The common $\Theta_\omega$ is a computable copy of 
$L_\omega$.  In particular there is a computable bijection between
the natural numbers and their representatives in 
$\Theta_\omega$.
\end{enumerate}
\end{prop}
\begin{proof}
We consider the domain 
of each $\Theta_\beta$ as a subset of $\mathbb N \times \mathbb N$.
For each infinite successor
notation $\beta\leq_\ast \alpha$, we reserve the column
$\mathbb N \times \{\beta\}$
for the elements of $\Theta_\beta \setminus \Theta_{\beta-1}$.

We proceed by effective transfinite recursion, and begin
with a computable presentation $\Theta_\omega$ 
of $L_\omega$, using $\mathbb N \times \{\omega\}$
as the domain, and choosing this presentation to satisfy 
the second niceness condition above.

Given $\alpha = \lim_n \alpha_n$ and $H_{\omega\cdot\alpha}$,
we define $\Theta_\alpha = \bigcup_n \Theta_{\alpha_n}$, 
which is uniformly computable from $H_{\omega\cdot\alpha}$
because the $n$th column of $H_{\omega\cdot\alpha}$ 
suffices to compute all atomic facts about $\Theta_\alpha$
involving elements from $\Theta_{\alpha_n}$.

Given $\alpha = \beta+1$ and $H_{\omega\cdot \alpha}$,
we can uniformly obtain $H_{\omega\cdot \beta +n}$
for each $n$.  Use $H_{\omega\cdot\beta}$ to obtain 
$\Theta_\beta$, and then add elements of 
$\mathbb N \times \{\alpha\}$ to the domain of $\Theta_\alpha$
as follows.  Let $(\phi_1,\bar z_1), (\phi_2,\bar z_2),\dots$ be some
canonical enumeration
of formula-parameter pairs (with the parameters in $\bar z$ drawn from 
$\Theta_\beta$) such that 
$$\operatorname{Def}(\Theta_\beta) = \left\{ \{ y  \in \Theta_\beta : \Theta_\beta\models \phi_i(y,\bar z_i)\} : i \in \omega\right\}$$

For each pair $(\phi_i, \bar z_i)$, ask
$H_{\omega\cdot \alpha}$ whether there is already some
$w \in \Theta_\beta$ such that for all $y \in \Theta_\beta$,
$$\Theta_\beta \models y \in w \iff \Theta_\beta \models \phi_i(y, \bar z_i).$$
Similarly ask if there is some $j<i$ such that for all $y \in \Theta_\beta$,
$$\Theta_\beta \models \phi_j(y, \bar z_j) \iff \Theta_\beta \models \phi_i(y, \bar z_i)$$
If either answer is yes, the defined set is already accounted for and can 
be ignored; if not, use a new element of $\mathbb N \times \{\alpha\}$
to represent a set with membership facts as above.
Because $\Theta_\beta$ is computable from $H_{\omega \cdot \beta}$
and all finite jumps of this set are available in $H_{\omega \cdot \alpha}$,
the latter can compute all these new facts. 
\end{proof}

\begin{proof}[Proof of Theorem \ref{thm:main2}]
(1) $\implies$ (2) is clear.  

(2) $\implies$ (3). If $T$ is a determined Borel code for $A$ in \HYP, then
the statement ``$f$ is a winning strategy for $X$ in $T$'' 
can be expressed in the language of set theory using only bounded quantifiers,
so both $A$ and $\HYP\setminus A$ are $\Sigma_1(L_{\omega_1^{ck}})$.

(3) $\implies$ (1).  Suppose that $A$ is $\Delta_1(L_{\omega_1^{ck}})$.  
Then there is a finite list of parameters
$\bar z \in L_{\omega_1^{ck}}$ and two $\Sigma_1$ 
formulas $\phi$ and $\psi$ such that for all $X \in 2^\omega$, 
$$X \in A \iff L_{\omega_1^{ck}}\models \phi(X,\bar z)  \text{   and  }  X\not\in A \iff L_{\omega_1^{ck}}\models \psi(X,\bar z).$$

We will define a completely determined Borel code for $A$ as follows.
Fix $\alpha \in \mathcal O^\ast\supseteq\mathcal{O}$.  We use decorations 
$\mathcal P = \{P_\beta : \gamma \leq_\ast \beta \leq_\ast \alpha\}$ and 
$\mathcal N = \{N_\beta : \gamma \leq_\ast \beta \leq_\ast \alpha\}$, where $\gamma$ is large enough that all elements of $\bar z$ are in $L_\gamma$.
We shall define $P_\beta$ to satisfy
$$|P_{\beta}| = \{X \in L_\beta : \beta \text{ is least such that } L_\beta \models \phi(X,\bar z)\}$$
and similarly for $N_{\beta}$ but using $\psi$.  
We now show how to computably
enumerate $\alpha$-ranked 
Borel codes for these sets $P_\beta$ and $N_\beta$, such that $P_\beta$
and $N_\beta$ each have rank $\omega\cdot\beta + O(1)$.

By the first niceness condition in Proposition \ref{prop:Lsays}, 
if $\beta\geq_\ast \gamma$,
then the elements of $\dom \Theta_\beta$ which represent the
parameters in $\bar z$ are in fact elements of $\dom \Theta_\gamma$ 
and do not depend on $\beta$. Therefore, without confusion we may also use 
the notation $\bar z$ to refer to those elements of $\dom \Theta_\gamma$
which represent the parameters $\bar z$ from $L_{\omega_1^{ck}}$.

Thus we have for all $X \in 2^\omega$ and $\beta\geq_\ast \gamma$,
$$ L_\beta \models \phi(X,z) \iff  \exists x \in \Theta_\beta[ x \text{ represents } X \text{ and } \Theta_\beta \models \phi(x, \bar z)]$$
The effective Borel complexity of 
``$\Theta_\beta\models \phi(x, \bar z)$'' is $\omega\cdot\beta + O(1)$,
with a constant that depends on $\phi$, specifically on the number
of quantifiers in $\phi$ (including bounded quantifiers, which will still require 
an unbounded search through $\dom \Theta_\beta$ in second order 
arithmetic).  
This is
because $H_{\omega\cdot\beta}$ uniformly computes the atomic diagram of 
$\Theta_\beta$, so the truth of $\phi(x,\bar z)$ is uniformly arithmetic 
in that diagram.

The effective Borel complexity of ``$x$ represents $X$'' is also 
$\omega\cdot\beta + O(1)$ using the second niceness condition 
in Proposition \ref{prop:Lsays}.  Let $h$ be a computable function 
such that $h(n) \in \dom \Theta_\omega$ represents the number $n$.
Then
$$``x \text{ represents } X\text{''} \iff \forall n \left[ X(n) = 1 \iff \Theta_\beta \models h(n) \in x\right].$$

Therefore, defining
$$|\hat P_\beta| := \{X \in 2^\omega : L_\beta \models \phi(X,\bar z)\}$$
we see this set has effective Borel complexity 
$\omega\cdot \beta + O(1)$.  Furthermore, the code $\hat P_\beta$ 
is obtainable and $\omega\cdot\beta+O(1)$-ranked,
uniformly in  $\beta$.  We define $\hat N_\beta$ 
similarly.  Then the desired decorations are
$$|P_{\beta}| := |\hat P_\beta| \setminus \left(\bigcup_{\delta<_\ast\beta} |\hat P_\delta|\right)$$
and similarly for $N_{\beta}$.  These decorations
are also uniformly $\omega\cdot\beta + O(1)$-ranked.

The computable procedure $\beta \mapsto P_{\beta}$ outlined above can 
also be applied to elements of $\mathcal O^\ast$, producing 
pseudo-ranked decorations for all $\beta <_\ast \alpha$.  
We apply 
Theorem \ref{thm:main_decoration} to the $(\omega\cdot \alpha)$-ranked
sets of decorations $\mathcal P$ and $\mathcal N$ constructed here.
The result is a completely determined
Borel code in \HYP{} which defines the set 
$A=\bigcup_{\beta \in \mathcal O} |P_{\beta}|$, as desired.
\end{proof}

\section{Applications}\label{sec:applications}

In light of Theorem \ref{thm:main2}, we can show that various sets have completely determined Borel codes in $\HYP$ by specifying an $\omega_1^{ck}$-recursive algorithm for computing them.  This allows us to know what \HYP{} believes 
about various theorems involving Borel sets.  We have selected some 
representative examples from a variety of areas.
The reader can surely supply many more 
examples than the ones given in this section.  

In this section we assume familiarity with 
$\alpha$-recursive computations; a reference is \cite{Shore1977}.
Theorem \ref{thm:main2} also shows that 
in $\HYP$, the determined Borel sets and the completely determined Borel sets 
coincide.  In this section, we simply use the terminology ``Borel'' to refer to this common 
concept.

\subsection{Well-Ordering and the Prisoner Hat Problem}


\begin{cor}\label{prop:wellordering}
  In \HYP, there is a Borel well-ordering of the universe.
\end{cor}
\begin{proof}
  We will associate hyperarithmetic reals $X\in 2^\omega$ with the value $o(X)=(\beta,e)$ where $\beta$ is least such that $X\leq_T\emptyset^\beta$ and $e$ is least such that $X=\phi_e^{\emptyset^\beta}$, and encode the ordering $X<Y$ if and only if $o(X)<o(Y)$, where $<$ is the lexicographic ordering on pairs.  Since $<$ is certainly a well-ordering, this will give the claim.

  On input $X,Y$, our algorithm can search for the first $\beta$ such that either $X\leq_T\emptyset^\beta$ or $Y\leq_T\emptyset^\beta$, and we can then check if $o(X)<o(Y)$ by checking an initial segment of the sets $\phi_e^{\emptyset^\beta}$ to see which of $X$ and $Y$ is computed first.
\end{proof}

Next recall the \emph{infinite prisoner hat problem}: we assume there is
a row of hat-wearing prisoners with order type $\omega$.  The hats can be red 
or blue.  The prisoners are facing toward the infinite end of the line,
so that each prisoner can see all the hat colors in front of them, but not 
their own hat color or the color of any previous hat.  The prisoners will 
be asked to name their own hat color, starting with the 0th prisoner
and going in order, so that each prisoner hears all the previous guesses.  
They win if they make one or fewer mistakes in total.  

It is well-known 
(see for example \cite{HardinTaylor2008}) that while the prisoners can 
win this game with the axiom of choice, 
there is no Borel winning strategy for them.  
But in $\HYP$, the
situation mirrors the real world and does so with the usual proof.

Formally, a Borel winning strategy 
for the prisoners is a Borel subset $B\subseteq 2^{<\omega}\times 2^\omega$.
A prisoner who hears the sequence $\tau \in 2^{<\omega}$
and sees the sequence $Y \in 2^\omega$ in front of them follows the 
strategy by guessing blue if $(\tau, Y) \in B$ and guessing red otherwise.

\begin{cor}
In $\HYP$, there is a Borel winning strategy 
for the prisoners in the infinite prisoner hat problem.
\end{cor}
\begin{proof}
By Corollary \ref{prop:wellordering},
as part of an $\omega_1^{ck}$-computation, we may search for the 
least real which has a given arithmetic property.

The strategy for the prisoners is then defined in the 
classical way, which we include for completeness.  
Each prisoner, hearing $\tau$ and seeing $Y$, 
begins by identifying the least 
real $X$ which agrees up to finitely many errors with $\tau\concat 0 \concat Y$.  
Since all prisoners use the same well-ordering,
they all identify the same $X$.  The 0th prisoner uses their 
guess to communicate the parity of errors between $X$ and the 
rest of the hats.  The $i$th prisoner, upon hearing the correct 
guesses of prisoners 1 through $i-1$, can then deduce their own 
hat color correctly by computing the parity of errors between 
$X$ and the hats they have seen and heard.  Observe that 
this prisoner strategy is $\omega_1^{ck}$-computable,
and thus Borel in $\HYP$.
\end{proof}

\subsection{Graphs}\label{subsection:graphs}

On the basis of the previous subsection, one might wonder if any 
construction that works by choice in the real world would work in a Borel way 
in $\HYP$.  The examples given in the next two examples show that this is not the case. 
Recall that a {\it $2$-coloring} of a graph $G = (V, E)$ is a function $c: V \to 2$ that assigns adjacent vertices to different colors. Classically, a graph has a $2$-coloring if and only if it has no odd cycles.  In second order arithmetic,
we consider graphs for which $V \subseteq 2^\omega$.  The graph 
$G$ is Borel if $V$ is Borel and $E$ is a Borel subset of $V \times V$.

\begin{prop}\label{prop:wo}
  In \HYP, there is a Borel acyclic graph with maximum degree $2$ which has no Borel 2-coloring.
\end{prop}
\begin{proof}
Fix $\alpha^\ast \in \mathcal O^\ast \setminus \mathcal O$.
  For each $\alpha<_\ast \alpha^\ast$ and $e\in \omega$, we fix two distinct computable reals $X_{\alpha,e,0}$ and $X_{\alpha,e,1}$. 

  We can describe a computation in stages indexed by $\beta\in\mathcal{O}$.  At the stage $\beta$, we decide all edges between pairs of reals $(X,Y)$ such that $\beta$ is least so that both $X$ and $Y$ are $\emptyset^\beta$-computable.

  We consider those $\alpha\leq\beta$ and those $e$ so that $\phi_e^{\emptyset^\alpha}$ appears to be a Borel code for a Borel 2-coloring, and $\beta$ is least so that $\emptyset^\beta$ computes evaluation maps for the colors of both $X_{\alpha,e,0}$ and $X_{\alpha,e,1}$ in $\phi_e^{\emptyset^\alpha}$.  For each such pair $\alpha,e$ we choose either one or two fresh reals Turing equivalent to $\emptyset^\beta$, and we add edges to create a path between $X_{\alpha,e,0}$ and $X_{\alpha,e,1}$ of length $2$ or $3$ (whichever is incompatible with the colors given to $X_{\alpha,e,0}$ and $X_{\alpha,e,1}$).  We place no other edges.
\end{proof}


Given $k \in \omega$, recall that a {\it $k$-edge-coloring} of a graph $G = (V, E)$ is a function $c: E \to k$ with the property that no two adjacent edges are assigned the same color. Vizing's Theorem states that if the maximum degree of the vertices in $G$ is $k$, for some $k \in \omega$, then $G$ has an edge coloring with at most $k+1$ colors (see, e.g., \cite[Theorem 5.3.2]{Diestel}). In the special case when $G$ has no odd cycles (i.e., when $G$ is bipartite), K\"onig showed that $G$ has a $k$-edge coloring (see \cite[Proposition 5.3.1]{Diestel}). On the other hand, Marks has shown \cite{Marks16} that there are $n$-regular acyclic Borel graphs with a Borel bipartition which require as many as $2n-1$ colors for a Borel edge coloring.

\begin{prop}\label{prop:edge1}
In \HYP, for every $k \geq 3$, there is a Borel acyclic graph with vertices of maximum degree $k$ with no Borel $(k+1)$-edge-coloring. 
\end{prop}
\begin{proof}
Let $N = \binom{k+1}{2}(k-1)+1$. (We have chosen $N$ so that when $N$ graphs are put into $\binom{k+1}{2}$ categories, some category contains at least $k$ graphs.) 
Fix $\alpha^\ast \in \mathcal O^\ast\setminus \mathcal O$.  For each $\alpha <_\ast \alpha^\ast$ and $e \in \omega$, we choose distinct computable reals $C_{\alpha, e}^1, \dots, C_{\alpha, e}^N$, $V_{\alpha, e}^1, \dots, V_{\alpha, e}^N$, and $W_{\alpha, e}^1, \dots, W_{\alpha, e}^N$. 

As in the proof of Proposition $\ref{prop:wo}$, we build a graph in stages $\beta \in \mathcal{O}$ so that at stage $\beta$, we determine all edges between pairs of reals $(X, Y)$, where $\beta$ is the smallest so that $\emptyset^\beta$ computes both $X$ and $Y$. 

At stage $\beta = 0$, for every $\alpha <_\ast \alpha^\ast$ and $e \in \omega$, and for $1 \leq i \leq N$, we connect $V_{\alpha, e}^i$ and $C_{\alpha, e}^i$ with an edge, and we connect $W_{\alpha, e}^i$ and $C_{\alpha, e}^i$ with an edge. Hence, for each $\alpha <_\ast\alpha^\ast$ and $e \in \omega$, we have $N$ disjoint paths of length two, each with a central `$C$' vertex and leaf vertices `$V$' and `$W$'. We will refer to this collection of $N$ paths as the {\it $(\alpha, e)$ computable subgraph}. 

At stage $\beta>0$, we handle all pairs $(\alpha, e)$, where $\alpha < \beta$ and $e \in \omega$, such that $\phi_e^{\emptyset^\alpha}$ appears to be a Borel code for a $(k+1)$-edge-coloring, and $\beta$ is the first ordinal after $\alpha$ so that $\emptyset^\beta$ computes evaluation maps for the color of every edge in the $(\alpha, e)$ computable subgraph. Given such a pair $(\alpha, e)$, we select a fresh vertex $X_{\alpha, e}$ that is Turing equivalent to $\emptyset^\beta$. We then find $k$ paths of length two in the $(\alpha, e)$ computable subgraph that all use the same two colors. For each of these paths, we connect the central `$C$' vertex to the new vertex $X_{\alpha, e}$. The given $(k+1)$-edge-coloring of the $(\alpha, e)$ computable subgraph cannot be extended to a $(k+1)$-edge-coloring of the extended graph, for $X_{\alpha, e}$ has degree $k$, and there are only $k-1$ colors available for its edges. 
\end{proof}


In Propositions $\ref{prop:wo}$ and $\ref{prop:edge1}$, the graph-builder has a source of power because the
graph-colorer is not able to wait to see all the neighbors of a given vertex.
If we restrict attention to connected graphs or to $d$-regular graphs, the graph-colorer may now have the upper hand.

\begin{prop}
  In \HYP, every connected Borel graph with no odd cycles has a Borel 2-coloring.
\end{prop}
\begin{proof}
  Let $E$ be a Borel code for the edges of the graph.
  
  Fix a real $X_0$.  At stage $\beta$ of our computation, we consider those $X$ such that $\beta$ is least so that there exist $X_0,\ldots,X_n\leq_T\emptyset^\beta$ with $X_n=X$ and evaluation maps $g_0,\ldots,g_{n-1}\leq_T\emptyset^\beta$ witnessing that $(X_i,X_{i+1})\in|E|$ for all $i<n$.

  We color $X$ by taking the first such path and coloring $X$ with $0$ if and only if $n$ is even.  Since the graph is assumed to be connected, each $X$ is colored at some stage $\beta$.  Since the graph has no odd cycles, this is a well-defined 2-coloring.
\end{proof}

For the rest of this section, $d\geq 1$ is any natural number. 

\begin{lemma}\label{lem:d-regular-component}
Suppose $G$ is a Borel $d$-regular graph in \HYP.
Then for every 
$X \in V(G)$, there is a computable ordinal $\beta$ such that 
$\emptyset^{\beta}$ computes an enumeration of the 
connected component of $X$ together with all evaluation maps 
needed to verify the component.\end{lemma}
\begin{proof}
  Observe that for each $X$, there are exactly $d$ neighbors, each hyperarithmetic, and, for each neighbor, a single evaluation map is needed to verify the edge, which is also hyperarithmetic.  So there is a unique least computable ordinal $\beta$ large enough 
  that $\emptyset^\beta$ computes $X$, all $d$ neighbors, and all $d$ evaluation maps witnessing the edges.
Similarly, for each distance $k$, there is a least $\beta$ such that $\emptyset^\beta$
 computes everything needed to enumerate and verify the set of vertices at distance at most $k$ from $X$.  Here is where it is used that 
$G$ is $d$-regular: for each $k$ this least $\beta$ can be recognized in a $\Sigma^1_1$ way.  Thus by $\Sigma^1_1$-bounding, there is some $\beta \in \mathcal O$ such that 
$\emptyset^\beta$ computes all vertices and edge-witnesses of the connected component 
of $X$.  With another couple of jumps, these vertices and witnesses can be 
enumerated in an organized way.  \end{proof}

\begin{prop}\label{prop:d-regular-bipartite}
  In \HYP, every Borel $d$-regular graph with no odd cycles has a Borel 2-coloring.
\end{prop}
\begin{proof}
  Each real in $X$ has a countable connected component in the given Borel graph.  
  In particular, if we are given a set $Y$ whose columns consist of all the path-neighbors of $X$ together with all the evaluation maps needed to verify them, we can verify in a hyperarithmetic way that it really is the entire connected component.
  By Lemma \ref{lem:d-regular-component}, if we search for such $Y$, we will find one.

  At stage $\beta$, we will color those $X$ such that $\beta$ is least so that $\emptyset^\beta$ computes an enumeration of the connected component of $X$ together with all evaluation maps needed to verify the component.

  When we find such an enumeration, we choose the one whose index (that is, the $e$ such that $\phi_e^{\emptyset^\beta}$ is the desired enumeration) is least, and color each $X$ in the component based on whether it has even distance to the vertex listed first in $\phi_e^{\emptyset^\beta}$.  Since the graph has no odd cycles, this is a well-defined 
  2-coloring.
\end{proof}

\begin{prop}\label{prop:edge2}
In \HYP, every Borel $d$-regular graph has a Borel $(d+1)$-edge-coloring. 
\end{prop}

\begin{proof}
Suppose $E$ is a Borel $d$-regular graph in \HYP. At stage $\beta$, we consider the connected components of $E$ for which $\beta$ is the least ordinal such that $\emptyset^\beta$ computes an enumeration $Y$ of the vertices in the component, 
together with all evaluation maps needed to verify the edges. (By Lemma $\ref{lem:d-regular-component}$, every connected component of $E$ will be handled at some stage $\beta$.) Given such a connected component $C$, we pick the least such 
enumeration $Y$  (the one given 
by the least $e$ such that the columns of $Y=\phi_e^{\emptyset^\beta}$ enumerate the 
component with all supporting evaluation maps). We use the ordering of the vertices of $C$ given by $Y$ to obtain a $\emptyset^\beta$-computable $(d+1)$-branching tree $T$, whose nodes represent partial $(d+1)$-edge-colorings of $C$. By Vizing's Theorem (see \cite[Theorem 5.3.2]{Diestel}), every finite induced subgraph 
of the component has a $(d+1)$-edge-coloring, so $T$ is infinite.
Therefore, by compactness,
$T$ has an infinite path. We use the left-most path (computable in $\emptyset^{\beta+1}$) to assign colors to the edges in $C$. 
\end{proof}

We finish out this section by showing that Marks' theorem for perfect matchings fails in \HYP.
Recall that given a graph $G$, a perfect matching is a subset $P \subseteq E(G)$ 
such that every vertex in the graph is an endpoint of exactly one edge from $P$.
Classically, a graph is bipartite if and only if it has no odd cycles.
A Borel bipartite graph is a Borel graph which has Borel 2-coloring to witness that it
has no odd cycles.

We need the following well-known fact, concerning the existence
of partial perfect matchings, but did not find a convenient reference,
so we also give a proof.  
\begin{lemma}\label{lem:induced-matching} 
If $G$ is any finite bipartite graph whose vertices have degree
at most $d$, there is some $E_0 \subseteq E(G)$ such that each 
vertex is an endpoint of at most one edge in $E_0$, and
each vertex of degree $d$ is an endpoint of exactly one edge in $E_0$.
\end{lemma}
\begin{proof}
Every finite $d$-regular bipartite graph has a perfect matching (see e.g. \cite[Corollary 2.1.3]{Diestel}).  So it suffices to show
that whenever $G$ satisfies the hypotheses of the lemma, then $G$
is an induced subgraph of some finite $d$-regular bipartite graph.  Let 
$V(G)=A_0 \cup B_0$ where $A_0$ and
$B_0$ witness that $G$ is bipartite.  By adding extra vertices to $G$
if necessary, we may assume without loss of generality that $|A_0|=|B_0|$.
If $G$ is already $d$-regular, we are done.  If $G$ is not $d$-regular,
we see that $|E(G)| < d|A_0|$.
Let $A_1$ and $B_1$ be new sets which each contain $k$ fresh vertices,
where $k\geq \max\{d|A_0| - |E(G)|, d\}$.  For each
vertex in $A_0$ which has fewer than $d$ neighbors,
connect it to some vertices in $B_1$ in order to bring its number 
of neighbors up to $d$.  Since $B_1$ contains enough vertices, this 
can be done in such a way that each vertex of $B_1$ receives 
at most one edge.  Similarly, add edges between $B_0$ 
and $A_1$ in order to bring the degree of each vertex in $B_0$ 
up to $d$ while adding at most one edge to each vertex of $A_1$.
Now exactly $d|A_0|-|E(G)|$ vertices in each of $A_1$ and $B_1$
 have an edge.  Add exactly one edge to each of the remaining 
 vertices of $A_1$ and $B_1$ by connecting them in pairs.  
 The problem is reduced to finding a $(d-1)$-regular 
 graph on the bipartition $\{A_1,B_1\}$ which does not use any
 of the existing edges between $A_1$ and $B_1$.  Since 
 $|A_1|=k>d-1$, such a graph exists.
\end{proof}

Now we can see the true situation with Borel perfect matchings  
differs from the situation in \HYP.

\begin{theorem}[Marks \cite{Marks16}]
For every $d>1$, there exists a Borel $d$-regular graph with no odd cycles which
has no Borel perfect matching.
Furthermore, this graph can be chosen to be acyclic and Borel bipartite.
\end{theorem}

\begin{prop}
In $\HYP$, every Borel $d$-regular graph with no odd cycles has a Borel perfect matching.
\end{prop}
\begin{proof}
 Given a Borel $d$-regular graph $E$ with no odd cycles,
at stage $\beta$ we consider those connected components
of $E$ for which $\beta$ is the least ordinal that computes an enumeration 
of the connected component, together with the sequence of evaluation maps needed 
to verify the component.

For each component, we fix the least enumeration $Y$ of that component.  
Using that enumeration to order the vertices, 
the set of perfect matchings for the component
can be given as a $\Pi^0_1(Y)$ class.  Now Lemma \ref{lem:induced-matching}
provides arbitrarily large partial perfect matchings, so compactness ensures 
that the $\Pi^0_1(Y)$ class is non-empty.
Now $\emptyset^{\beta+1}$ can compute its 
leftmost perfect matching, which we apply to the connected component being considered.

By Lemma \ref{lem:d-regular-component}, every component of $E$ will eventually be 
found and a perfect matching computed on it.
\end{proof}

Since the theories of hyperarithmetic analysis are among the weakest axioms 
strong enough to make sense of Borel sets, the fact that Borel sets in \HYP{} 
do not act like the real-world ones is not too surprising.
But it does establish the theories of hyperarithmetic analysis as
reasonable base theories, when asking if theorems
proved by Borel Determinacy in \cite{Marks16}
could be proved by measure or category methods.

In particular, we would be curious to know if Marks' theorem that 
there is a $d$-regular acyclic Borel graph with no Borel $d$-coloring 
follows from $\CDPB$ or $\mathsf{CD}\text-\mathsf M$.  
Here $\mathsf{CD\text-M}$ is
the principle ``every completely determined Borel set is measurable'' 
(see \cite{Westrick20}).  
One might suspect these theories are too weak, 
based on the following result of
Conley, Marks \& Tucker-Drob: for $d\geq 3$, 
every $d$-regular acyclic Borel graph has a measurable $d$-coloring and a 
$d$-coloring with the property of Baire, regardless of which Borel
measure or which Polish Borel-compatible 
topology is used on the vertex set \cite[Theorem 1.2]{ConleyMarksTuckerDrob2016}.
This shows that if the theorem can be proved by measure 
or category, the proof cannot proceed in ``the usual way'' of showing 
that there is no measurable or Baire measurable coloring.  However,
there remains the possibility 
that measure or category is used in some creative way in an 
alternate proof,
for example by being applied to some object other than the 
purported $d$-coloring.  On the other hand, it is not known whether 
this theorem can even be proved in second order arithmetic.

\subsection{Borel Dual Ramsey Theorem}

We recall the statement of the Borel Dual Ramsey Theorem.  First, we need some notation.
\begin{definition}
  For $k\in\mathbb{N}\cup\{\omega\}$, $(\omega)^k$ is the set of partitions of $\omega$ into exactly $k$ nonempty pieces.  When $p\in(\omega)^\omega$, we write $(p)^k$ for the set of coarsenings of $p$ into exactly $k$ blocks.
\end{definition}

The Borel Dual Ramsey Theorem says:
\begin{quote}
  For all finite $k,\ell\geq 1$, if $(\omega)^k=C_0\cup\cdots\cup C_{\ell-1}$ where each $C_i$ is Borel then there exists $p\in(\omega)^\omega$ and an $i<l$ such that $(p)^k\subseteq C_i$.
\end{quote}

\begin{theorem}
  In \HYP, the Borel Dual Ramsey Theorem fails.
\end{theorem}
\begin{proof}
  We show this even with $k=\ell=2$.

  Given $p\in(\omega)^\omega$ with $p=\bigcup_i p_i$ and a monotone function $f$, let us define $f(p)\in(\omega)^2$ so that $f(p)=q_0\cup q_1$ where $q_1=\bigcup_ip_{f(i)}$ and $q_0=\omega\setminus q_1$.  By a \emph{finite modification} of $f(p)$, we mean $f(p)=q_0\cup q_1$ where $q_1=\bigcup_{i\geq n}p_{f(i)}$ and $q_0=\omega\setminus q_1$.  The important properties are that the finite modifications are pairwise distinct and whenever $q$ is a finite modification of $f(p)$, $q\leq_T f\oplus p$ and $f\leq_T q\oplus p$.

    For each $\beta$, let $f_\beta$ be a monotone function Turing equivalent to $\emptyset^{\beta+1}$ and which is eventually larger than every function computable from $\emptyset^\beta$.

  Let $p^0_\beta,\ldots,p^n_\beta,\ldots$ enumerate those elements of $(\omega)^\omega$ such that $\beta$ is least with $p^i_\beta\leq\emptyset^\beta$.  We recursively choose, for each $p^i_\beta$, two elements $q^{i,0}_\beta,q^{i,1}_\beta\in(\omega)^2$ by letting $q^{i,0}_\beta$ be the first finite modification of $f_\beta(p^i_\beta)$ distinct from all $q^{j,b}_\beta$ with $j<i$ and $q^{i,1}_\beta$ the first finite modification of $f_\beta(p^i_\beta)$ distinct from all $q^{j,b}_\beta$ and also $q^{i,0}_\beta$.

  Observe that if $q^{i,b}_\beta=q^{i',b'}_{\beta'}$ then $\beta=\beta'$, and therefore $i=i'$ and $b=b'$: if $\beta'<\beta$ then $q^{i',b'}_{\beta'}\leq_T f_{\beta'}\oplus p^{i'}_{\beta'}\leq_T\emptyset^{\beta'+1}$, while $\emptyset^{\beta+1}\leq_T f_\beta\leq_T p^i_\beta\oplus q^{i,b}_\beta$ and, since $p^i_\beta\leq_T\emptyset^\beta$, we must have $q^{i,b}_\beta\not\leq_T\emptyset^\beta$.

  By construction, for each $\beta$, the $q^{n,b}_\beta$ can be uniformly enumerated by $\emptyset^{\beta+k}$ for some $k$ large enough to carry out these computations.  So at stage $\beta+k$, we color all the $q^{n,0}_\beta$ with color $0$ and all other elements of $(\omega)^2$ which are computable from $\emptyset^{\beta+1}$ which have not already been colored with color $1$.

  For any $p\in(\omega)^\omega\cap\HYP$, we have $p=p^n_\beta$ for some $n,\beta$, and we have $q^{n,0}_\beta\in C_0$ and $q^{n,1}_\beta\in C_1$, so $(p)^2\not\subseteq C_0$ and $(p)^2\not\subseteq C_1$.  Therefore the Borel Dual Ramsey Theorem fails in \HYP.
\end{proof}

\bibliographystyle{alpha}
\bibliography{references}

\end{document}